\documentclass[12pt]{article}

\usepackage[bookmarks,bookmarksnumbered,%
    citebordercolor={0.8 0.8 0.8},filebordercolor={0.8 0.8 0.8},%
    linkbordercolor={0.8 0.8 0.8},%
    urlbordercolor={0.8 0.8 0.8},%
    pagebackref,linktocpage%
    ]{hyperref}

\usepackage[margin=1.0in]{geometry}
\usepackage{graphicx}              
\usepackage{amsmath}               
\usepackage{amsfonts}              
\usepackage{amsthm}                
\usepackage{url}
\usepackage{amssymb}
\usepackage{latexsym}
\usepackage{bm}
\usepackage{enumitem}

\newcommand{\Ss}{\section}

\newcommand{\GSM}{\cite{M1}}

\newcommand{\LFD}{\cite{M3}}
\newcommand{\DUM}{\cite{M4}}

\newcommand{\factories}{3.2}
\newcommand{\phin}{5.1}

\newcommand{\GCS}{5.2}

\newcommand{\itemsaiphone}{3.1}
\newcommand{\itemsaiphoneB}{3.1}
\newcommand{\gowersnorms}{Proposition 3.1}
\newcommand{\fouriernorm}{Lemma 2.1}
\newcommand{\fouriertransform}{Proposition 2.3}
\newcommand{\phinone}{Lemma 5.5}

\newcommand{\tri}{\Delta}
\newcommand{\tn}[2]{T[#1,#2]}
\newcommand{\tcal}[2]{\Delta(#1 , #2 )}

\allowdisplaybreaks

\newtheorem{theorem}{Theorem}[section]
\newtheorem{thm}[theorem]{Theorem}
\newtheorem{prop}[theorem]{Proposition}
\newtheorem{lemma}[theorem]{Lemma}
\newtheorem{definition}[theorem]{Definition}

\theoremstyle{definition}

\let\oldmarginpar\marginpar
\renewcommand\marginpar[1]{\-\oldmarginpar[\raggedleft\footnotesize #1]%
{\raggedright\footnotesize #1}}

\def\keywords#1{\bigskip \par\noindent{\it Keywords and phrases: }#1\par}
\def\AMS#1{\par\noindent{\it 2010 Mathematics Subject Classification: }#1\par}

\DeclareMathOperator{\R}{\mathbb{R}}
\DeclareMathOperator{\C}{\mathbb{C}}
\DeclareMathOperator{\T}{\mathbb{T}}
\DeclareMathOperator{\N}{\mathbb{N}}
\DeclareMathOperator{\Z}{\mathbb{Z}}

\DeclareMathOperator{\eps}{\epsilon}

\DeclareMathOperator{\ra}{\rightarrow}
\DeclareMathOperator{\half}{\frac{1}{2}}

\title{Higher-order Fourier dimension and frequency decompositions}
\author{Marc Carnovale}
\date{\today}
\begin{document}

\maketitle


\begin{abstract}
This paper continues work begun in \cite{M1}, in which we introduced a theory of Gowers uniformity norms for singular measures on $\R^d$. There, given a $d$-dimensional measure $\mu$,
we introduced a $(k+1)d$-dimensional measure $\tri^k\mu$, and developed a uniformity norm $\|\mu\|_{U^k}$ whose $2^k$-th power is equivalent to $\tri^k\mu([0,1]^{d(k+1)}$. 
In the present work, we introduce a fractal dimension associated to measures $\mu$ which we refer to as the $k$th-order Fourier dimension of $\mu$. This $k$-th order Fourier dimension is 
a normalization of the asymptotic decay rate of the Fourier transform of the measure $\int \tri^k\mu(x;\cdot)\,dx$, and coincides with the classic Fourier dimension in the case that $k=1$.
It provides quantitative control on the size of the $U^k$ norm. The main result of the present paper is that this higher-order Fourier dimension controls the rate at which $\|\mu-\mu_n\|_{U^k}\rightarrow 0$,
where $\mu_n$ is an approximation to the measure $\mu$. This allows us to extract delicate information from the Fourier transform of a measure $\mu$ and the interactions of its frequency
components, which is not available from the $L^p$ norms- or the decay- of the Fourier transform. In future work \cite{M4}, we apply this to obtain a differentiation theorem for singular measures.
\end{abstract}

\tableofcontents

\section*{Acknowledgments}

Many thanks to Prof.'s Izabella Laba and Malabika Pramanik for introducing me to the question of arithmetic progressions in fractional sets which motivated this work, for supporting
me while I did much of the work in this paper,
for the sharing of their expertise in Harmonic Analysis, for funding my master's degree, and much more. 

Thanks to Nishant Chandgotia for his forgiveness of the mess I made of our office for two years and his bountiful friendship.

Thanks to Ed Kroc, Vince Chan, and Kyle Hambrook for the frequent use of their time and ears and their ubiquitous encouragement.


\noindent \keywords{Gowers norms, Uniformity norms, singular measures, Finite point configurations, Salem sets, Hausdorff dimension, Fourier dimension}
\vskip0.2in

\noindent \AMS{28A78, 42A32, 42A38, 42A45, 11B25, 42B10, 42B25}


\section{Introduction}\label{ch:Introduction}

This paper continues work begun in \cite{M1}. There, for any measure $\mu$ on the $d$-dimensional torus we introduced the $k+1$-dimensional measure $\tri^k\mu$, 
a singular analogue of the object $\tri^k f$ relevant in the definition of Gowers' uniformity norms, $\|\cdot\|_{U^k}$, from additive combinatorics. 

In the discrete setting, these uniformity norms provide a notion of pseudorandomness by acting as a measure of the extent to which a function $f$ on, say, $[1,N]$, correlates with $k-1$st degree (phase) polynomials,
and are useful because they encode arithmetical properties of $f$ in the following sense: roughly, it is not difficult to show that appropriate control on the $U^k$ norm of a function guarantees that that function's support must contain many $k+1$-term
arithmetic progressions $a,a+b,\dots,a+kb$, as well as other linear patterns. 

The $U^2$ norm of a function is identical to the $L^4$ norm of that function's Fourier transform. Also connected to the decay of an object's Fourier transform is the notion of
\textit{Fourier dimension} $\dim_{\mathbb{F}}\mu$  of a measure $\mu$ from Geometric Measure Theory. For $\mu$ a measure on $\R^d$, define $\dim_{\mathbb{F}}\mu$ to be 
\begin{align*}
 \dim_{\mathbb{F}}\mu:=\sup\{\beta\in[0,d] : |\widehat{\mu}(\xi)|\lesssim (1+|\xi|)^{-\frac{\beta}{2}}\}
\end{align*}

In \cite{laba}, it was shown that a singular measure on $\R$ with a sufficiently large \textit{Fourier dimension} must in its support contain $3$-term progressions, but the case of
higher-term progressions was left wide-open as it is not amenable to the spectral methods employed there. The purpose of the present paper is to present a higher-order generalization of
the Fourier dimension condition inspired by the relationship enjoyed between the notions of Fourier dimension and the $U^2$ norm, and to develop the technology we need in the forthcoming
paper \cite{M3} in order to exploit this 
higher-order Fourier dimension so as to demonstrate progressions and other linear patterns in $\R^d$.

One interpretation of the present work is the following. Many problems in Harmonic Analysis have appeared to rely crucially on the curvature properties of the set under question, though in
recent years it has been found that these problems have analogues where the notion of curvature no longer seems to be present (\cite{mitsis},\cite{mocken}, \cite{LabaDiff}). Further, several
problems in Additive Combinatorics about the arithmetic structure of sets are known to be determined by the large values of the Fourier transform, which in a continuous context corresponds
to the large values at asymptotic frequencies.
The earlier paper \cite{M1} and the present work together attempt to suggest a new framework in which to attack such problems for sets in $\R^d$.
In particular, the notion of higher-order Fourier decay presented here
requires either cancellation in the Fourier transform of a measure or large sets of small values, to which traditional direct Fourier decay and $L^p$ methods are insensitive. As
an application of this viewpoint, we use the main result of the present paper to prove in \cite{M4} that measures with sufficiently large $k-1$-st order Fourier dimension differentiate (in
the sense of the Hardy-Littlewood differentiation theorem)
$L^p$ for $p'<k$.

\subsection[Review of Gowers norms for singular measures]{Review of \cite{M1}}

In \cite{M1}, we defined the $U^k$ norm for a measure $\mu$ and showed it to be equivalent to
\begin{align}
 \|\mu\|_{U^k} = \tri^k\mu(\T^{k+1})^{\frac{1}{2^k}}
\end{align}
We verified that this does indeed define a norm, and agrees with the norm $\|f\|_{U^k}$ Gowers introduced for functions in \cite{gowers}.

Given $\mu_{\iota}\in U^k$, $\iota\in\{0,1\}^k$, we furhter introduced, for $\bm{\mu}:=\left\{\mu_{\iota}\right\}_{\iota\in\{0,1\}^k}$, the measure on $\T^{k+1}$

\begin{align}
 \tri^k(\bm{\mu}) := \tri^k(\left\{\mu_{\iota}\right\}_{\iota\in\{0,1\}^k})
\end{align}
and showed that it exists whenever each $\mu_{\iota}\in U^k$, and we proved a Gowers-Cauchy-Schwarz Inequality

\begin{align}
  \left|\tri^k(\bm{\mu})(\T^{k+1})\right|:=&
  \left|<\bm{\mu}>\right|\\\leq& \prod_{\iota\in\{0,1\}^k} \|\mu_{\iota}\|_{U^k}
\end{align}

Setting $\bm{\mu}_i = \left\{\mu_{i\iota'}\right\}_{i\iota'\in\{0,1\}^{k-1}}$ for $i=0,1$, we defined $\tri^k(\bm{\mu}_0,\bm{\mu}_1) = \tri^k\bm{\mu}$, and showed that

\begin{align}\label{5}
 \widehat{\tri^{k}(\bm{\mu_0},\bm{\mu_1})}(\xi;\bm{\eta}) = \sum_{c\in\Z^{k-1}}  \widehat{\tri^{k-1}(\bm{\mu_0})}(-\eta_{k};\bm{\eta}'-c) \widehat{\tri^{k-1}(\bm{\mu_1})}(\xi+\eta_{k};c)
\end{align}

Note that if $\mu_{\iota} \equiv \nu$ for each $\iota$, then $ \tri^k(\bm{\mu}) :=\tri^k\nu$ and also that $\|\mu\|_{U^k}^{2^k} = \widehat{\tri^{k}\mu}(0;0)$.

For $k\in\N$, extend the definition of $\|\mu\|_{U^{k+1}}$ to be infinity when $\tri^k\mu$ is undefined. 
Define $U^{k+1}$ to be the space of all finite measures $\mu$ on $\mathbb{T}^d$ for which $\|\left|\mu\right|\|_{U^{k+1}} < \infty$.
Then the following theorem is a rephrasing
of part of Theorem 1.2 from \cite{M1}.

\begin{thm}\label{thm:the following identities}
	Let $\mu$ be a measure on $\mathbb{T}^d$. Then for all $k$, the finite measure $\tri^{k+1}\mu$ exists if and only if $\mu\in U^{k+1}$.

	Further, the following identities hold for all $\mu\in U^{k+1}$
	\begin{align}\label{iphone}
	  &\widehat{\tri^{k+1}\mu}(\xi;\bm{\eta}) = \sum_{\bm{c}\in\Z^k} \widehat{\tri^k\mu}(-\eta_{k+1};\bm{\eta}'-\bm{c}) \widehat{\tri^k\mu}(\xi+\eta_{k+1};\bm{c})\\&\label{ipad}
	  \|\mu\|_{U^{k+1}}^{2^{k+1}}= \sum_{\bm{c}\in\Z^k}|\tri^k\mu(0;\bm{c})|^2
	\end{align}
	
	\end{thm}

For definitions, the reader should refer to \cite{M1}.

In this paper, we continue our study of measures with finite $U^k$ norm, and single out a particularly nice geometric class of such measures: those with what we term here positive $k-1$st order
Fourier dimension. 

Although control over $U^k$ norms suffices to present progressions in the discrete context, in the continuous setting more quantitative control is necessary. 
For $3$-term progressions, this is naturally provided by the Fourier dimension of the measure $\mu$. Indeed, since it is the $U^2$ norm which controls 3-AP's
and $\|\mu\|_{U^2}^4 = \sum_{\xi\in\Z} |\hat{\mu}(\xi)|^4$, the decay rate of $\hat{\mu}$ provides finer information than the $U^2$ norm.

However, this Fourier dimension will not suffice to obtain sufficient quantitative control in the case of higher-term patterns. Since 
$\|f\|_{U^3} = \sum_{\bm{\eta}} |\widehat{\tri^2 f}(0;\eta_1,\eta_2)|^2$, and one may compute (or use Theorem \ref{thm:the following identities} to find) that 
$\widehat{\tri^2 f}(0;\bm{\eta}) = \sum_c \hat{f}(\eta_1-c)\hat{f}(\eta_1-\eta_2 +c) \hat{f}(c)\hat{f}(\eta_2-c)$, if no information beyond that $|\hat{f}(\bm{\eta})|\leq |\bm{\eta}|^{-\frac{\beta}{2}}$ 
with $\beta<1$ is hypothesized, then we just fall short of obtaining useful decay information for $\widehat{\tri^2 f}(0;\bm{\eta})$ and obtain no information whatsoever about $\widehat{\tri^3 f}$.

Instead, since 
\begin{align}
 \|\mu\|_{U^k} = \sum_{\bm{\eta}\in\Z^{k-1}} |\widehat{\tri^{k-1}\mu}(0;\bm{\eta})|^2
\end{align}
preserving the relationship between the uniformity norm and Fourier dimension we make the following definition 

\vspace{5pt}
\hrulefill
\vspace{2pt}
\begin{definition}
 \emph{For $k>1$, we define the $k$th-order Fourier dimension of a measure $\mu$ on $\R^d$ to be the supremum over all $\beta\in(0,d)$ for which}
 \begin{align}
|\widehat{\tri^i\mu}(0;\bm{\eta})|\lesssim(1+|\bm{\eta}|)^{-\frac{i+1}{2}\beta}
 \end{align}
\emph{ for all $i\leq k$.}
 
\emph{If $\mu$ is a measure with nontrivial compact support on $\mathbb{T}^d$, then we identify it with a measure on $\R^d$ in the natural way in order to define its higher-order Fourier dimension.}

\emph{We further say that the measure $\mu$ possesses a $k$th order Fourier decay of $\beta$ if for all $i\leq k$,}
\begin{align}
|\widehat{\tri^i\mu}(0;\bm{\eta})|\leq C(1+|\bm{\eta}|)^{\frac{i+1}{2}\beta}
\end{align}
 \end{definition}
 
\vspace{2pt}
\hrulefill
\vspace{5pt}

Note that since $\widehat{\tri^k\mu}$ can be computed purely in terms of $\widehat{\mu}$ by Proposition \ref{thm:the following identities}, we need not assume that $\tri^k\mu$ exist in 
order to define the $k$-th order Fourier dimension. However, 
using the methods of \cite{M1},
it is not hard to show that positive $k$th order Fourier dimension implies the
existence of $\tri^k\mu$.

The statement that $\mu$ have a $k$-th order Fourier dimension of $\beta < d$ is just the statement that the measure $\tri^k\mu$ on $\R^{d(k+1)}$ have a classical Fourier dimension of
$(k+1)\beta < (k+1)d$. If in analogy with the additive combinatorial $\tri^k f$ we think of $d\tri^k\mu(x;u)$ as telling us about the size of the intersection of the measure $\mu$  with the shifts $\mu^{\iota\cdot u}$, $\iota\in\{0,1\}^k$ of 
itself  around the point $x$, then an assumption of higher order Fourier dimension can be thought of as the assertion that distribution of where these shifts are largest is fairly 
``dispersed'', in that it doesn't correlate too strongly with any high frequency $e^{-2\pi i \bm{\eta}\cdot}$. So while classical Fourier dimension tells us that the mass of a measure is 
in some sense distributed ``fairly,'', higher order Fourier dimension tells us that the set of distances between areas of large density are themselves distributed evenly, as well as as the distances between these distances, and so on. 
From this perspective, it is the appropriate generalization of the relationship between the $U^2$-norm and the linear Fourier dimension. 

One might ask whether such decay assumptions are possible aside from the trivial case of Lebesgue measure. 
In future work, we pursue an affirmative answer.


Higher-order Fourier dimension gives us quantitative control that the $U^k$ norm does not in the following sense. 

Let $\phi_n$ be an approximate identity with Fourier transform $\widehat{\phi_n}$ essentially supported in the ball $B(0,2^{n+1})$.

Further, set $\mu_n = \phi_n\ast\mu$.

In this paper, we seek to show that the $k$-th order Fourier dimension of $\mu$ gives control on the size of 
\begin{align}\label{fiffy}
 \|\mu-\mu_n\|_{U^{k+1}}
\end{align}
and in fact provides a convergence rate depending on this Fourier dimension so that (\ref{fiffy}) is summable in $n$.

Inspired by Theorem \ref{thm:the following identities}, we introduce the decomposition

\begin{align}
\|\mu\|_{U^{k+1}}^{2^{k+1}} =& \|\mu\|_{U^{k+1},<N}^{2^{k+1}} + \|\mu\|_{U^{k+1},\geq N}^{2^{k+1}}\\
 \|\mu\|_{U^{k+1},<N} :=& \left(\sum_{\bm{\eta}\in\Z^k} \widehat{\phi_N^{[k]}}(0;\bm{\eta}) |\widehat{\tri^k\mu}(0;\bm{\eta})|^2\right)^{\frac{1}{2^{k+1}}}
 \\ \|\mu\|_{U^{k+1},\geq N} :=& \left( \|\mu\|_{U^{k+1}}^{2^{k+1}} -\|\mu\|_{U^{k+1},<N}^{2^{k+1}}\right)^{\frac{1}{2^{k+1}}}
\end{align}

The work in this paper comes down to controlling these expressions.


\subsection{Results}


\subsubsection{Outline}
Let  $k\in\N$, $\mu$ be a measure on $\mathbb{T}^d$, $(\phi_n)$ an approximate identity, $\mu_n:=\phi_n^{\ast^{k}}\ast\mu$ be $k$-copies of $\phi_n$ convolved with $\mu$.

In Section \ref{ch:Norm Decomposition}, we describe the main decomposition and tool used in this paper; namely, a splitting of the norm $\|\mu\|_{U^k}$ into the part coming from
the low frequencies of $\tri^{k-1}\mu$, $\|\mu\|_{U^k,\leq N}$ and the part coming from the high frequencies of $\tri^{k-1}\mu$, $\|\mu\|_{U^k,>N}$, where $N$ is some large parameter. 
We also obtain a ``monotonicity result'', Lemma \ref{thm:smalltri}, which bounds  $\|\mu_n\|_{U^k,>N}-\|\mu\|_{U^k,>N}$ by $\|\mu\|_{U^k,\leq N } -\|\mu_n\|_{U^k,\leq N}$, and which we
will need in order to leverage control on $\|\mu\|_{U^k,>N}$, and $\|\mu-\mu_n\|_{U^k,\leq N}$  into control on $\|\mu_n\|_{U^k,>N}$ (under a higher-order Fourier decay assumption on $\mu$) when
we prove the main result in the following section.

The main result of this paper is  Proposition \ref{thm:rk}, which, supposing that the measure $\mu$ on $\R^d$ possesses a $k$-th order Fourier dimension close enough to $d$
,  gives  a bound on $\|\mu-\mu_n\|_{U^{k+1}}$ which decays  exponentially in $n$. It is in Section \ref{ch:lemmas} that we put together the pieces from the rest of the paper
in order to obtain this result.

The easy ingredient in the main result occurs in Section \ref{ch:Higher-order Fourier Dimension}, where we show how to use control on $\|\mu\|_{U^k}$ in order to obtain control on $\|\mu\|_{U^{k+1},\leq m}$ (Lemma \ref{thm:start}). 
This allows for the induction used to prove the main result in the previous section. We also formally define the concept of higher-order Fourier dimension here.

In Section \ref{ch:FrequencyRestricted}, we establish, via an analogue of a Gowers-Cauchy-Schwarz Inequality, an analogue of a triangle inequality for $\|\mu\|_{U^k,>N}$, of the form
\begin{align*} 
\|f+g\|_{U^k,>N}\leq C\max(\|f\|_{U^k},\|g\|_{U^k})\max(\|f\|_{U^k,>N},\|g\|_{U^k,>N})
\end{align*}

This is the important final piece to proving the main result; when combined with Lemma \ref{thm:smalltri} and Lemma \ref{thm:start}, and an assumption of higher-order Fourier decay,
Proposition \ref{thm:rk} follows without much work. This is also the difficult part of the proof.

In Section \ref{ch:Truncations}, we prove the various identities necessary to derive the results of Section \ref{ch:FrequencyRestricted}.

The primary ingredient in obtaining the results of Section \ref{ch:FrequencyRestricted}, however, comes from Section \ref{ch:Spatial}, where we perform the Fourier transform calculation
needed to represent one of the key objects in Setion \ref{ch:Truncations} on the spatial side, which we need in order to obtain bounds on it in Section \ref{ch:FrequencyRestricted}.

\subsubsection{Overview of the Approach}
The main result of this paper is  Proposition \ref{thm:rk}, which, given a measure $\mu$ on $\R^d$ of $k$-th order Fourier dimension close enough to $d$ and setting $\mu_n=\phi_n\ast\mu$ for 
an approximate identity $\phi_n$,  gives  a bound on $\|\mu-\mu_n\|_{U^{k+1}}$ which decays  exponentially in $n$.

We prove Proposition \ref{thm:rk} by induction on $k$. When $k=2$, it holds immediately by the identity $\|\mu\|_{U^2} = \|\widehat{\mu}\|_{L^4}$. 

According to Theorem \ref{thm:the following identities}, we have that
$\|g\|_{U^{k}}^{2^{k}} = \widehat{\tri^k g}(0;0)$ and $\|g\|_{U^{k+1}}^{2^{k+1}}= \sum_{\bm{\eta}\in\Z^{k}} |\widehat{\tri^k g}(0;\bm{\eta})|^2$. Assuming Propostion \ref{thm:rk} holds for some $k$, we have that $\widehat{\tri^k (\mu-\mu_n)}(0;0) = \|\mu-\mu_n\|_{U^{k}}^{2^k}$ is exponentially small in $n$;
since for all $\bm{\eta}$,
$|\widehat{\tri^k (\mu-\mu_n)}(0;\bm{\eta})|\leq \widehat{\tri^k (\mu-\mu_n)}(0;0)$ by Lemma \ref{thm:start}, this says that 

\begin{align}\label{Fact 1}
 \|\mu-\mu_n\|_{U^{k+1},\leq 2^m}^{2^{k+1}} :\approx\sum_{\bm{\eta}\in\Z^{k},|\bm{\eta}|_{\infty}\leq 2^m} |\widehat{\tri^ k (\mu-\mu_n)}(0;\bm{\eta})|^2\leq 2^m (\widehat{\tri^k (\mu-\mu_n)}(0;0))^2
\end{align}
is exponentially small in $n$ (for small enough $m$ which depends on $n$). (Fact 1)

So Proposition \ref{thm:rk} follows from getting an exponentially decaying bound on

\begin{align}
 \|\mu-\mu_n\|_{U^{k+1},>2^m}^{2^{k+1}} :\approx\sum_{\bm{\eta}\in\Z^{k},|\bm{\eta}|_{\infty} >2^m} |\widehat{\tri^ k (\mu-\mu_n)}(0;\bm{\eta})|^2
\end{align}
for some not too-large $m$.

The assumption of $k$-th order Fourier dimension guarantees that $\|\mu\|_{U^{k+1},>2^m}$ is exponentially decaying in $m$. We may also show, based on Fact 1 above, that
$\|\mu_n\|_{U^{k+1},>2^m}$ is exponentially decaying in $m$ (think of Fact 1 as the statement that $\mu$ is about $\mu_n$ as far as $\|\cdot\|_{U^{k+1},\leq 2^m}$ is concerned;
then a sort of ``monotonicity'' result (Lemma \ref{thm:smalltri}) tells us that $\mu_n$ is similarly about the same as $\mu$ as far as $\|\cdot\|_{U^{k+1},>2^m}$ is concerned.)

Modulu an appropriate choice of $m$, what is left is to combine the exponential decay of $\|\mu\|_{U^{k+1},>2^m}$ and $\|\mu_n\|_{U^{k+1},>2^m}$ into exponential decay of $\|\mu-\mu_n\|_{U^{k+1},>2^m}$. 
This is asking for a sort of triangle inequality for $\|\cdot\|_{U^{k+1},>2^m}$, but unfortunately $\|\cdot\|_{U^{k+1},>2^m}$ is not a norm and a direct triangle inequality is not a available.
Instead, we have Proposition \ref{thm:reltri}, a weighted triangle inequality. As in the case of the triangle inequality for $\|\cdot\|_{U^{k+1}}$, this follows from expanding out the expression
for $\|a+b\|_{U^{k+1},>2^m}$ into a sum of products of terms involving $a$'s and $b$'s, and applying a sort of Gowers-Cauchy-Schwarz to each such product (Lemma \ref{thm:highcross}).

The proof of Lemma \ref{thm:highcross} is a straightforward sequence of calculations on the Fourier side, together with a bound on the physical side. In detail, we want to take 
$\sum_{|\bm{\eta}|>2^m} |\widehat{\tri^k (f_1,f_2)}(0;\bm{\eta})|^2$, $f_i\in \{a,b\}^{2^{k-1}}$, and bound it by an expression with half as many cross terms,
say  
\begin{align}
       C(\sum_{|\bm{\eta}|>2^m} |\widehat{\tri^kf_1,f_1}|^2)^{\frac{1}{2}}(\sum_{|\bm{\eta}|>2^m} |\widehat{\tri^kf_2,f_2}|^2)^{\half}
      \end{align}
 This is not something we know to do, but we do show the bound
$C (\sum_{|\bm{\eta}|>2^m} |\widehat{f_1,f_1}|^2)^{\frac{1}{4}}$, and rearranging $(f_1,f_1)$ into the form $(f_1',f_2')$ , where $f_i'\in\{a,b\}^{2^{k-1}}$ has half as many cross terms 
as either of the original $f_i$, (which doesn't affect the sum, by 
Lemma \ref{thm:permute}) and iterating yields Lemma \ref{thm:highcross}.

The argument to obtain the bound of $C (\sum_{|\bm{\eta}|>2^m} |\widehat{f_1,f_1}|^2)^{\frac{1}{4}}$  is: writing 
\begin{align}
 \widehat{\tri^{k+1}_{s_j>0}(f_1,f_2)}(0;0):\approx \sum_{\bm{\eta}\in\Z^k, |\eta_j|>2^m} |\widehat{\tri^k (f_1,f_2)}(0;\bm{\eta})|^2
\end{align}
we have

\begin{align}
 \sum_{|\eta_j|>2^m} |\widehat{\tri^k (f_1,f_2)}(0;\bm{\eta})|^2 = \sum_{\bm{\eta}\in\Z^k} \widehat{\tri^k_{s_j>N} (f_1,f_1)}(0;\bm{\eta})\widehat{\tri^k (f_2,f_2)}(0;\bm{\eta})
\end{align}
(Lemma \ref{thm:CS11}), which by Cauchy-Schwarz is bounded by 
\begin{align}
 C[ \sum_{\bm{\eta}\in\Z^k} |\widehat{\tri^k_{s_j>2^m} (f_1,f_1)}(0;\bm{\eta})|^2]^{\half}
\end{align}
which by Lemma \ref{thm:CS12} is 

\begin{align}
 C[\sum_{|\eta_j|>2^m} \widehat{\tri^k_{s_j>N,s_j+\eta_j>2^M} (f_1,f_1)}(0;\bm{\eta})\widehat{\tri^k (f_2,f_2)}(0;\bm{\eta})]^{\half}
\end{align}
where $\widehat{\tri^k_{s_j>2^m,s_j+\eta_j>2^m} (f_1,f_1)}(0;\bm{\eta})$ is given by the same sum as $\widehat{\tri^k_{s_j>N} (f_1,f_1)}(0;\bm{\eta})$,
but restricting the sum to be only over those $c_j$ so that $|c_j+\eta_j|>2^m$. 

Applying Cauchy-Schwarz to this gives us

\begin{align}
 C\bigg[[\sum_{|\eta_j|>2^m} |\widehat{\tri^k (f_2,f_2)}(0;\bm{\eta})|^2\sum_{|\eta_j|>2^m} |\widehat{\tri^k_{s_j>N,s_j+\eta_j>2^M} (f_1,f_1)}(0;\bm{\eta})|^2]^{\half}\bigg]^{\half}
\end{align}
and the first sum here is of the same form as what we started with, except that it will have either twice as many $a'$s and half as many $b$'s as what we started with, or the reverse (which is what we wanted so that after finitely many applications of this process, we end up with $\|a\|_{U^{k+1},>2^m}^{2^{k+1}}$ or $\|b\|_{U^{k+1},>2^m}^{2^{k+1}}$. 
So Lemma \ref{thm:highcross} is reduced to showing that the second sum above is bounded (actually, it likely should exhibit some decay, but we do not know how to take advantage of this). We bound the second sum by

\begin{align}
 \sum_{} |\widehat{\tri^k_{s_j>N,s_j+\eta_j>2^M} (f_1,f_1)}(0;\bm{\eta})|^2
\end{align}
Lemma \ref{thm:overgrowth} allows us to write this sum in physical space. 
If we set aside the technicalities owing to the fact that the $a$ and $b$ may be measures, let $\widehat{\psi}\approx 1_{|\cdot|>2^m}$, and write $f_1=(g_1,g_2)$, $g_i\in\{a,b\}^{k-2}$, 
we would have that the sum is the same as

\begin{align}
 &\int  \tri^{k-2}g_1(x;\bm{u}'') \tri^{k-2}g_2(x-(u_{k-1}-t-a);\bm{u}'')\tri^{k-2}g_1(x-u_{k};\bm{u}'') \tri^{k-2}g_2(x-(u_{k-1}-t-b)-u_k;\bm{u}'')\\\cdot&
 \tri^{k-2}g_1(x;\bm{u}'') \tri^{k-2}g_2(x-(u_{k-1}-a-t;\bm{u}'')\tri^{k-2}g_1(x-u_{k};\bm{u}'') \tri^{k-2}g_2(x-(u_{k-1}-a-s)-u_k;\bm{u}'')\\\cdot&
 \psi(a)\psi(b)\psi(s)\psi(t)\,dt\,ds\,da\,db\,dx\,du
\end{align}

First bounding this by replacing the $\psi$'s by $|\psi|$'s, by applying Cauchy-Schwarz we may disentangle the $|\phi(t)|$ so that instead of hitting both the $u_{k-1}$'s on the top line and only one of the $u_{k-1}$'s
on the bottom line, it hits all of the $u_{k-1}$'s, and so may be integrated out, and similarly the effect of the other $|\psi|$'s may be removed. This is the point of Lemma \ref{thm:overgrowth}, but the proof is
greatly complicated by the fact that we are dealing with measures, and so must take care not only to smooth everything with mollifiers, but also that when we take limits after the application of Cauchy-Schwarz,
we reassemble the measures $\tri^{k-2} g_1\cdots \tri^{k-2}g_2$ into $\tri^k (g_1,g_2)$ (actually, in this lemma, we combine the result following Cauchy-Schwarz 
into $\tri^k (g_1,g_1)\tri^k (g_2,g_2)$).



\Ss{A Decomposition of the Norm}
\label{ch:Norm Decomposition}

Throughout the rest of this section, fix $\phi=(\phi_n)_{n\in\N}$, an approximate identity on $\mathbb{T}^d$ such that $\phi_n$  has compact support on the Fourier side for each $n\in\N$,
and such that $\widehat{\phi_n}\approx 1_{B(0,2^n)}$.
Further, set $\phi_n^c = 1-\phi_n$.

For $\mu\in U^k$ and $N\in\N$, define

\begin{align*}&
 \|\mu\|_{U^k,\leq N} = \left(\sum_{\bm{\eta}\in\Z^{k}} |\widehat{\phi_N^{[k]}}(0;\bm{\eta})|^2|\widehat{\tri^k\mu}(0;\bm{\eta})|^2\right)^{\frac{1}{2^k}}\\&
 \|\mu\|_{U^k,> N} = \left(\sum_{\bm{\eta}\in\Z^{k}} |(1-\widehat{\phi_N^{[k]}})(0;\bm{\eta})|^2|\widehat{\tri^k\mu}(0;\bm{\eta})|^2\right) ^{\frac{1}{2^k} }
\end{align*}

The first inequality of Corollary \phin applied to $\mu_{\iota}\equiv\mu$ tells us that if we truncate the high frequencies of $\mu$ by convolving with $\phi_n^{\ast^k}$ to obtain $\mu_n$, 
the resulting measure $\tri^{k+1}\mu_n$
cannot have greater Fourier mass than the truncation to low frequencies of $\tri^{k+1}\mu$.

The following tells us that any surplus mass exhibited by  $\tri^k\mu_n$ restricted to an annlus in frequency space,
is bounded by the deficit of the Fourier mass of $\tri^k\mu_n$ on the ball that annulus surrounds.

\begin{lemma}\label{thm:smalltri}
 Let $\mu$ belong to $U^{k+1}$. If for some $\eps$, $\|\mu\|_{U^{k+1}, <m}^{2^{k+1}} =\|(\ast^{k+1} \phi_n)\ast\mu\|_{U^{k+1}, <m}^{2^{k+1}} + \eps$, then
 \begin{align}
 \|\mu\|_{U^{k+1},>m}^{2^{k+1}}  \geq \|(\ast^{k+1}\phi_n)\ast\mu\|_{U^{k+1},> m}^{2^{k+1}}-\eps
 \end{align}

\end{lemma}
\begin{proof}
Suppose $\mu$, $\eps$ as above.

 We may write 
 \begin{align}\label{sewe1}
  \|\mu\|_{U^k,\leq m}^{2^{k+1}}  + \|\mu\|_{U^{k+1},> m}^{2^{k+1}} = \|\mu\|_{U^{k+1}}^{2^{k+1}} 
 \end{align}
 By Corollary \phin, 
 \begin{align}
\|(\ast^{k+1}\phi_n)\ast\mu\|_{U^{k+1}}^{2^{k+1}} \leq \int |\int \phi_n^{[k]}\ast\tri^k\mu(x;u)\,dx|^2\,du
\end{align}
which using Plancherel, we write on the Fourier side as
\begin{align}
 \sum_{\bm{\eta}\in\Z^{k}} |\widehat{\phi_n^{[k]}}(0;\bm{\eta})|^2|\widehat{\tri^k\mu}(0;\bm{\eta})|^2=\|\mu\|_{U^{k+1},<n}^{2^{k+1}}\leq \|\mu\|_{U^{k+1}}^{2^{k+1}}
\end{align}
So
\begin{align}\label{sewe2}
 (\ref{sewe1})\geq \|(\ast^{k+1}\phi_n)\ast\mu\|_{U^{k+1}}^{2^{k+1}} 
 \end{align}

 At the same time,
\begin{align}\label{sewe3}
 \|(\ast^{k+1}\phi_n)\ast\mu\|_{U^{k+1}}^{2^{k+1}} =& \|(\ast^{k+1}\phi_n)\ast\mu\|_{U^{k+1},\leq m }^{2^{k+1}} + \|(\ast^{k+1}\phi_n)\ast\mu\|_{U^{k+1},>m }^{2^{k+1}}\\\notag{}=&
 \|\mu\|_{U^{k+1},\leq m }^{2^{k+1}} - \eps+ \|(\ast^{k+1}\phi_n)\ast\mu\|_{U^{k+1},>m }^{2^{k+1}}
 \end{align}
by hypothesis.

Using (\ref{sewe2}), which states that $(\ref{sewe1})\geq (\ref{sewe3})$, we have shown that

\begin{align}
 \|\mu\|_{U^{k+1},\leq m }^{2^{k+1}}+  \|\mu\|_{U^{k+1},>m }^{2^{k+1}}\geq \|\mu\|_{U^{k+1},\leq m }^{2^{k+1}} - \eps+ \|(\ast^{k+1}\phi_n)\ast\mu\|_{U^{k+1},>m }^{2^{k+1}}
\end{align}
and so we are done.
\end{proof}


\Ss{The main theorem}
\label{ch:lemmas}

Throughout this section, $\mu$ will refer to a measure in $U^k(\T)$ of $k$-th order Fourier decay of $\beta$, so that
$|\widehat{\tri^k\mu}(\zeta)|\leq C_{\mu}|\zeta|^{-\frac{k+1}{2}\beta}$,
$\phi_n=2^{n}\phi(2^n\cdot)$ for $\phi$ such that $(\phi_n)$ forms an approximate identity on $\mathbb{T}^d$, 
and $\mu_n:=(\ast^k \phi_n)\ast\mu$. Constants will be assumed to depend on $\C_{\mu}$ and $\phi_n$ unless stated otherwise.

Here we develop the quantitative estimate which shows how one may put to use an assumption of higher-order Fourier decay to control the terms in a decomposition of the $U^k$ norm, 
and in particular how determine how quickly $\mu_n\ra\mu\in U^k$. This is naturally useful when dealing with measures with good higher-order Fourier dimensions, and will be
the primary ingredient in \LFD and \DUM.

\begin{prop}\label{thm:rk}
 With $\mu$ and hypotheses as above, and setting
 \begin{align}
  r_k:= \bigg(\prod_{j=3}^k \left[2-\frac{{2^{3j-2}}}{2^{3j-2}-[1-\frac{(j+1)\beta}{jd}]}\right]\bigg) (2\beta-d)
 \end{align}
we have the bound

\begin{align}
 \|\mu-\mu_n\|_{U^k}\leq C 2^{-\frac{r_k}{2^k}n}
\end{align}
where the constant depends only on the choice of $\phi_n$, and the  constant $C_F$ appearing in 
\begin{align}
 |\widehat{\tri^j\mu}(0;\eta)|\leq C_F |\eta|^{-(j+1)\frac{\beta}{2}}, \hspace{5pt} j=1,\dots, k
\end{align}

\end{prop}

\begin{proof}
 Our proof is inductive. The $k=2$ case is immediate since
 $\|\mu-\mu_n\|_{U^2}^4 \approx \sum_{|c|>2^n}|\widehat{\mu}(c)|^4\leq C \sum_{|c|>2^n} |c|^{-2\beta} \approx 2^{(2\beta-d)n}$.
 (It is here that orthogonality between $\mu-\mu_n$ and $\mu_n$is crucial, and for this reason  we think of the methods of Sections \ref{ch:FrequencyRestricted} and \ref{ch:Truncations}
 as substituting for this orthogonality for higher $U^k$.)
 
 Suppose that the claim holds for some particular $k$. Set $f=\mu_n$ and $g=\mu-\mu_n$. Then we have that
 \begin{align}
  \widehat{\tri^k g}(0;0) = \|g\|_{U^k}^{2^k}\leq C 2^{-r_kn}
\end{align}
Since $|\widehat{\tri^k g}(0;\eta)|\leq \widehat{\tri^k g}(0;0)$, we have by Lemma \ref{thm:start} that 

\begin{align}
 |\widehat{\tri^k g}(0;\eta)|\leq 2^{-r_k n}
\end{align}
for all $\eta$.

Let $m>0\in\R$. Then

\begin{align}
 \|g\|_{U^k+1,<m}^{2^{k+1}}\leq \sum_{\eta\in\Z^k} |\widehat{\phi_{m}^{[k]}}(0;\eta)|^2|\widehat{\tri^{k} g}(0;\eta)|^2\lesssim 2^{dkm}2^{-2r_kn}=2^{-2r_kn+dkm}
\end{align}

By the reverse triangle inequality applied to $\|\cdot\|_{U^{k+1},<m}$, we have

\begin{align}
&|\|\mu\|_{U^{k+1},<m}-\|(\ast^k\phi_n)\ast\mu\|_{U^{k+1},<m}|\\=&
 |\|f+g\|_{U^{k+1},<m}-\|f\|_{U^{k+1},<m}|\leq \|g\|_{U^{k+1},<m}\\\leq&
 2^{\frac{-2r_kn+dkm}{2^{k+1}}}
\end{align}

We will use that

\begin{align}
 \|\mu\|_{U^{k+1},>m}^{2^{k+1}}\\\lesssim&
 \sum_{|\eta|>2^m} |\widehat{\tri^k\mu}(0;\eta)|^2\\\lesssim&
 \sum_{|\eta|>2^m} |\eta|^{-(k+1)\beta}\\\lesssim&
 2^{-((k+1)\beta-kd)m}
\end{align}

By Lemma \ref{thm:smalltri}, we then have that 
\begin{align}
 &\|f\|_{U^{k+1},>m}^{2^{k+1}}\leq \|\mu\|_{U^{k+1},>m}^{2^{k+1}}+2^{-2r_kn+kdm}\\\lesssim&
 2^{-((k+1)\beta-kd)m}+2^{-2r_kn+kdm}
\end{align}

Applying Proposition \ref{thm:reltri}, we then have that
\begin{align}
 \|g\|_{U^{k+1},>m}\leq C \big( 2^{-((k+1)\beta-kd)m}+2^{-2r_kn+kdm}\big)^{\frac{1}{2^{3k-2}}}
\end{align}

Hence 

\begin{align}\label{aboveagain}
 \|g\|_{U^{k+1}}^{2^{k+1}}=\|g\|_{U^{k+1},\leq m}^{2^{k+1}}+\|g\|_{U^{k+1},\leq m}^{2^{k+1}}\\\lesssim&
2^{-2r_kn+kdm}+ \big( 2^{-((k+1)\beta-kd)m}+2^{-2r_kn+kdm}\big)^{\frac{1}{2^{3k-2}}}
\end{align}

Up to constants, this is minimized when we choose $m=\frac{2^{3k-2}}{(k+1)\beta+(2^{3k-2}-1)kd} r_kn$, at which time the exponents $-[((k+1)\beta-kd)m](2^{3k-2})$ and $-2r_kn+kdm$ 
are both equal. Plugging this value of $m$ in to (\ref{aboveagain}), we obtain, up to constants, the bound

\begin{align}
 \|g\|_{U^{k+1}}\lesssim 2^{-\left[2-\frac{{2^{3k-2}}}{2^{3k-2}-[1-\frac{(k+1)\beta}{kd}}]\right] r_k n} = 2^{-r_{k+1}n}
\end{align}
So the claim is completed by induction.

\end{proof}

It is worth noting that $r_k=r_k(\beta)$ increases as $\beta$ increases and is positive for $\beta$ close enough to $d$, as one would expect.


\Ss{Low Frequency Components}
\label{ch:Higher-order Fourier Dimension}

Our goal in this section is to prove Lemma \ref{thm:start}, which tells us that $\widehat{\tri^k\mu}(0;0) = \|\mu\|_{U^{k-1}}^{2^{k-1}}$ controls $\widehat{\tri^k\mu}(0;\bm{\eta})$ for any $\bm{\eta}$. Informally, this
means that the sum of the small frequencies of $\tri^k\mu$ is controlled by the a constant multiple (depending on our definition of ``small'') of $\|\mu\|_{U^{k-1}}^{2^{k-1}}$. 
      
In order to prove Lemma \ref{thm:start}, we will need to know the affect of permuting the measures $\mu_{\iota}$ on $\tri^{k+1}\bm{\mu}$. This we record as Lemma \ref{thm:permute} at the 
end of the present section.

\begin{lemma}\label{thm:start}
 Let $\mu\in U^{k}$ be a signed measure and suppose that $k\geq 2$. Then for any $\bm{\eta}\in\Z^k$, 
 \begin{align}
  |\widehat{\tri^k\mu}(0;\bm{\eta})|\leq \widehat{\tri^k\mu}(0;0)
 \end{align}
 \end{lemma}
 \begin{proof}
 We first claim that for any $\bm{\varepsilon}\in\Z^k$, we have
 
 \begin{align}\label{claim}
    |\widehat{\tri^k\mu}(0;\bm{\varepsilon})|\leq |\widehat{\tri^k\mu}(0;0,\dots,0,\varepsilon_k)|
 \end{align}
 
 To see this, let $\bm{\varsigma}$ be any vector in $\Z^k$. We use Corollary \factories of \GSM to write
  \begin{align}\label{thq1}
   \widehat{\tri^k\mu}(0;\bm{\varsigma}) = \lim_{n\ra\infty} \int e^{-2\pi i \bm{\varsigma}\cdot u} \Phi_{n}\ast\tri^{k-1}\mu(x-u_k;\bm{u}')\overline{\Phi_{n}\ast\tri^{k-1}\mu(x;\bm{u}')}\,dx\,d\bm{u}
 \end{align}
 
 Changing variables by sending $u_{k}\mapsto -u_{k}+x$, this becomes
 
  \begin{align}\label{thq2}
    (\ref{thq1})&=\lim_{n\ra\infty} \int e^{-2\pi i \bm{\varsigma}'\cdot \bm{u}'}e^{-2\pi i \varsigma_k(x-u_k)} \Phi_{n}\ast\tri^{k-1}\mu(u_k;\bm{u}')\overline{\Phi_{n}\ast\tri^{k-1}\mu(x;\bm{u}')}\,dx\,du_k\,d\bm{u}'\\\notag{}=&
    \lim_{n\ra\infty} \int e^{-2\pi i \bm{\varsigma}'\cdot \bm{u}'} |\int e^{-2\pi i\varsigma_k x}\Phi_{n}\ast\tri^{k-1}\mu(u_k;\bm{u}')\,dx|^2\,d\bm{u}'.
 \end{align}
 
 This shows that
 \begin{align} \label{claim1}
\widehat{\tri^k\mu}(0;\bm{\varsigma}) = \lim_{n\ra\infty} \int e^{-2\pi i \bm{\varsigma}'\cdot \bm{u}'}|\int e^{-2\pi i\varsigma_k x}\Phi_{n}\ast\tri^{k-1}\mu(u_k;\bm{u}')\,dx|^2\,d\bm{u}'.
 \end{align}

Letting $\bm{\varsigma} = \bm{\varepsilon}$ and taking absolute values in (\ref{claim1}), we then have
\begin{align}\label{thq3}
|\widehat{\tri^k\mu}(0;\bm{\varepsilon})|\leq \lim_{n\ra\infty} \int |\int e^{-2\pi i\varepsilon_k x}\Phi_{n}\ast\tri^{k-1}\mu(x;\bm{u}')\,dx|^2\,d\bm{u}'
 \end{align}
 
 Now using (\ref{thq1}) and (\ref{thq2}) again with $\bm{\varsigma}= (0,\dots,0,\varepsilon_k)$, the right-hand side of (\ref{thq3}) is
 
 \begin{align}\label{thq4}
=& \lim_{n\ra\infty} \int e^{-2\pi i \bm{\varsigma}\cdot u}\Phi_{n}\ast\tri^{k-1}\mu(x-u_k;\bm{u}')\overline{\Phi_{n}\ast\tri^{k-1}\mu(x;\bm{u}')}\,dx\,d\bm{u}\\\notag{}=&
\widehat{\tri^k\mu}(0;\bm{\varsigma}) = \widehat{\tri^k\mu}(0;0,\dots,0,\varepsilon_k) 
 \end{align}
since we chose $\bm{\varsigma}=(0,\dots,0,\varepsilon_k)$.

So we have shown 

 \begin{align} \label{claim2}
|\widehat{\tri^k\mu}(0;\bm{\varepsilon})| \leq& \widehat{\tri^k\mu}(0;0,\dots,0,\varepsilon_k)
\end{align}
which was the claim.

Now let $\bm{\eta}\in \Z^k$ be a vector. By the claim, we have that 

 \begin{align} \label{claim3}
|\widehat{\tri^k\mu}(0;\bm{\eta})| \leq& \widehat{\tri^k\mu}(0;0,\dots,0,\eta_k)
\end{align}

Since $k\geq 2$, the last entry  of the vector $(0,\dots,0,\eta_k)$ is distinct from the first entry, so when we use Lemma \ref{thm:permute} to interchange them, we can write the
right-hand side of (\ref{claim3}) as $\widehat{\tri^k\mu}(0;\eta_k,0,\dots,0)$ where this does in earnest have a last entry of $0$. 

Now letting $\bm{\varepsilon}= (\eta_k,0,\dots,0)$ and applying the claim (\ref{claim}) once more, we then have that
 \begin{align} \label{claim4}
|\widehat{\tri^k\mu}(0;\eta_k,0,\dots,0)| \leq& \widehat{\tri^k\mu}(0;0,\dots,0,0)
\end{align}
and combining this with (\ref{claim3}), we have
 \begin{align} \label{claim5}
|\widehat{\tri^k\mu}(0;\bm{\eta})| \leq& \widehat{\tri^k\mu}(0;0)
\end{align}
which is what we sought to show.
\end{proof}

\begin{lemma}\label{thm:permute}
Let $k\in\N$,  $\mu_{\iota}$, $\iota\in\{0,1\}^k$ be measures on $\T^d$, $\bm{\mu}= \{\mu_{\iota}\}_{\iota\in\{0,1\}^{k+1}}$. 
For $\pi$ any permutation of $\{1,\dots,k\}$ and $k$-tuple $(x_1,\dots,x_k)$, let $\pi (x_1,\dots,x_k) = (x_{\pi 1},\dots,x_{\pi k})$. Define 
\begin{align*}
 \bm{\pi\mu} =& \left\{\mu_{\pi\iota}\right\}_{\iota\in\{0,1\}^{k+1}}\\
 \pi\tri^{k+1}\bm{\mu} =& \tri^{k+1}(\bm{\pi\mu})
\end{align*}

Define also $\bm{\pi\mu_0} = \{\mu_{\pi0\iota'}\}_{\iota'\in\{0,1\}^k}$, and  $\bm{\pi\mu_1} = \{\mu_{\pi1\iota'}\}_{\iota'\in\{0,1\}^k}$.

Then 
\begin{align}\label{simlye}
 \widehat{\pi\tri^{k+1}(\bm{\mu})}(\xi;\bm{\eta}) = \widehat{\tri^{k+1}(\bm{\mu})}(\xi;\pi\bm{\eta}) 
\end{align}

Further, 

\begin{align}\label{similarlye}
 \widehat{\tri^{k+1}_{}(\bm{\mu})}(\xi;\bm{\eta}) = \widehat{\tri^{k+1}(\bm{\mu})}(\xi;\bm{\eta}',\xi-\eta_{k+1}) 
\end{align}
\end{lemma}

\begin{proof}
(\ref{similarlye}) is an immediate consequence of the change of variables $x\mapsto x + u_{k+1}$, so we turn to (\ref{simlye}).

 Since the claim trivially holds for $k+1=1,2$, we will assume it holds for some $k$ and show that it then holds for $k+1$ as well.
 
 Proposition \itemsaiphone tells us that
 
 \begin{align}
  \widehat{\tri^{k+1}(\bm{\mu})}(\xi;\bm{\eta})  = \sum_{\bm{c}\in\Z^k}
  \widehat{\tri^{k}(\bm{\mu_1})}(\xi+\eta_{k+1};\bm{c})  \overline{\widehat{\tri^{k}(\bm{\mu_0})}(\eta_{k+1};\bm{c}-\bm{\eta}') }
 \end{align}

 Thus our inductive hypothesis applied to the terms in the sum gives us the claim for $\pi$ any permutation of $[1,\dots,k+1]$ fixing $k+1$. So it suffices to show the claim for $\pi$ the transposition $(k,k+1)$.

 We compute on the Fourier side via two applications of Proposition \itemsaiphoneB that
  \begin{align*}
  &\widehat{\tri^{k+1}(\bm{\mu})}(\xi;\bm{\eta})  = 
  \sum_{\bm{c}\in\Z^k} \widehat{\tri^k\bm{\mu}_0}(\xi+\eta_{k+1};\bm{c}) \overline{\widehat{\tri^k\bm{\mu}_1}(\eta_{k+1};\bm{c}-\bm{\eta}')}  \\=&
\sum_{\bm{c}} \bigg(\sum_{\bm{a}\in\Z^{k-1}} \widehat{\tri^{k-1}(\bm{\mu_{00}})}(\xi+\eta_{k+1}+c_k;\bm{a})\overline{\widehat{\tri^{k-1}(\bm{\mu_{01}})}(c_{k};\bm{a}-\bm{c}')}\bigg) \\&
\bigg(\sum_{\bm{b}\in\Z^{k-1}} \overline{\widehat{\tri^{k-1}(\bm{\mu_{10}})}(\eta_{k+1}-\eta_k+c_k;\bm{b})}\widehat{\tri^{k-1}(\bm{\mu_{11}})}(c_k-\eta_{k};\bm{b}-\bm{c}'+\bm{\eta}'')\bigg)
   \end{align*}
   and that each of these series converges uniformly.
   Uniform convergence in $\bm{c}'$ means it is valid to interchange the order of summation. 
   We do so, then  send $\bm{b}\mapsto -\bm{b}+\bm{a}$ and $\bm{c}'\mapsto -\bm{c}'+\bm{a}$ so this becomes

        \begin{align*}&
\sum_{,\bm{a}\in\Z^k,\bm{c}\in\Z^{k-1},\bm{b}\in\Z^{k-1}}\bigg( 
\widehat{\tri^{k-1}(\bm{\mu_{00}})}(\xi+\eta_{k+1}+c_k;\bm{a})\overline{\widehat{\tri^{k-1}(\bm{\mu_{01}})}(c_{k};\bm{a}+\bm{c}'+\bm{a})}\bigg)\\&
\bigg(\overline{\widehat{\tri^{k-1}(\bm{\mu_{10}})}(\eta_{k+1}-\eta_k+c_k;-\bm{b}+\bm{a})}\widehat{\tri^{k-1}(\bm{\mu_{11}})}(c_k-\eta_{k};-\bm{b}+\bm{a}+\bm{c}'-\bm{a}+\bm{\eta}'')\bigg) \\=&
\sum_{(\bm{b},c_k)}\bigg(\sum_{\bm{a}\in\Z^{k-1}} \widehat{\tri^{k-1}(\bm{\mu_{00}})}(\xi+\eta_{k+1}+c_k;\bm{a})\overline{\widehat{\tri^{k-1}(\bm{\mu_{10}})}(\eta_{k+1}-\eta_k+c_k;\bm{a}-\bm{b})}\bigg)\\&
\bigg(\sum_{\bm{c}'\in\Z^{k-1}}\overline{\widehat{\tri^{k-1}(\bm{\mu_{01}})}(c_{k};\bm{c}')}\widehat{\tri^{k-1}(\bm{\mu_{11}})}(c_k-\eta_{k};\bm{c}'-\bm{b}+\bm{\eta}'')\bigg)
   \end{align*}

%

Applying again two iterations of Proposition \itemsaiphoneB to collect together terms, this becomes

\begin{align}
  & \sum_{(a,c_k)\in\Z^k} \widehat{\tri^{k}(\bm{\pi\mu_1})}(\xi+\eta_{k};a,c_k)      \overline{\widehat{\tri^{k}(\bm{\pi\mu_0})}(-\eta_{k};a-\bm{\eta}'',c_{k}-\eta_{k+1})}  \\=&
  \widehat{\tri^{k+1}(\bm{\pi\mu})}(\xi;\eta_1,\dots,\eta_{k-1},\eta_{k+1},\eta_k)
 \end{align}
  This completes the induction, and we have shown that
 \begin{align}
   \widehat{\tri^{k+1}(\bm{\pi\mu})}(\xi;\bm{\eta}) =  \widehat{\tri^{k+1}(\bm{\mu})}(\xi;\pi\bm{\eta})
 \end{align}
 \end{proof}


\Ss{Frequency-Restricted Cauchy-Schwartz and Triangle Inequalities}\label{ch:FrequencyRestricted}

The purpose of this section is to establish certain variants of the (Gowers)-Cauchy-Schwarz inequality and the triangle inequality which control how the different frequencies of a measure 
$\mu$ may interact to affect the frequencies of $\tri^k\mu$ and the $U^{k+1}$ norm of $\mu$. These will provide a substitute for orthogonality in $U^k$, which is needed since two functions on
$\mathbb{T}^d$ with disjoint frequency supports may still interact in $U^k$ in complicated ways.
  In more detail, the main result says that if the portion of the $U^{k+1}$ norm of $\mu_1$ and $\mu_2$ coming from high frequencies are both small,
 in the sense that for $i=1,2$, $\|\mu_i\|_{U^{k+1},>N}$ is small for some $N$, then so too is any ``cross-term'' $\tri^{k+1}_{>N}(\{\mu_{\iota}\}_{\iota\in\{0,1\}^{k+1}})(\T^{k+2})$ where $\mu_{\iota}\in\{\mu_1,\mu_2\}$.
 It goes a way towards amending the failure of a true Cauchy-Schwarz result for $\tri^{k+1}_{>N}$.

Given $\mu_{\iota}\in U^k$, $\iota\in\{0,1\}^k$, $N\in\N$, and $j\in[1,k]$ an integer, let $\tri^k_{s_j>N}(\bm{\mu})$ and $\tri^k_{s_j>N,\eta_j+s_j>N}(\bm{\mu})$ be defined
as in Section \ref{ch:Truncations}. The results of Section \ref{ch:Truncations} will be used throughout this section.

 \begin{lemma}\label{thm:highcross}
Suppose that $\mu_{\iota}\in\{\mu_1,\mu_2\}$ for all $\iota\in\{0,1\}^k$, with $\|\mu_i\|_{U^{k+1}}\leq A$ and $\|\mu_i\|_{U^{k+1},>N} <\eps$ for $i=1,2$. Then
\begin{align}
 \sum_{\bm{\bm{\eta}}\in\Z^k,\|\bm{\bm{\eta}}\|_{\infty}>N} |\widehat{\tri^k(\bm{\mu})}(0;\bm{\bm{\eta}})|^2\\=:&
 \tri^{k+1}_{s_{\infty}>N}(\bm{\mu},\bm{\mu})
 \leq C(A)\eps^{{2^{-2(k-1)}}}\label{thirteynineone}
\end{align}

 \end{lemma}

 \begin{proof}
  We can write  $\sum_{\|\bm{\bm{\eta}}\|_{\infty}>N} \leq \sum_{j=1}^k \sum_{\bm{\bm{\eta}},s_j>N}$, so it suffices to check that such a bound as above holds for one such sum.
  
  Our proof will be by induction on the number $l$ of components $\iota_n$ of $\iota=(\iota_1,\dots,\iota_k)$ upon which $\mu_{\iota}$ depends.
  
  If $l=0$, then $\mu_{\iota}=\mu_{\kappa}$ for every $\iota,\kappa\in\{0,1\}^k$ so that $\mu_{\iota}=\mu_1$ or $\mu_2$ for every $\iota$, and so the sum in 
  question is $\|\mu_i\|_{U^{k+1},>N}\leq \eps$ by hypothesis.
  
  So suppose that 
  
  \begin{align}\label{thirteyninenine}
 \sum_{\bm{\bm{\eta}}\in\Z^k,s_j>N} |\widehat{\tri^k(\bm{\nu})}(0;\bm{\bm{\eta}})|^2\leq C(A)\eps^{{2^{-2(l-2)}}}
\end{align}
($\bm{\nu}=\{\nu_{\iota}\}_{\iota\in\{0,1\}^k}$) whenever both $\nu_{\iota}\in\{\mu_1,\mu_2\}$ and $\nu_{\iota}$ is such that whether $\nu_{\iota}=\mu_1$ or $\mu_2$ depends on only $l-1$ components of $\iota$.

Now assume that $\mu_{\iota}\in\{\mu_1,\mu_2\}$ depends on at most $l$ components of $\iota$.

Let $\pi$ be a permutation so that $\mu'_{\iota}:=\mu_{\pi \iota}$ depends on at most the first $l$ components (but not the remaining $k-l$). Write $\bm{\mu'} := \left\{\mu_{\pi\iota}\right\}_{\iota\in\{0,1\}^k}$.

Then

\begin{align}
& \sum_{\bm{\bm{\eta}},s_j>N} |\widehat{\tri^k(\bm{\mu})}(0;\bm{\bm{\eta}})|^2\\\notag{}=&
 \sum_{\bm{\bm{\eta}},s_j>N} |\widehat{\tri^k(\bm{\mu'})}(0;\pi\bm{\bm{\eta}})|^2\\\notag{}=&
   \sum_{\bm{\bm{\eta}},s_{\pi j}>N} |\widehat{\tri^k(\bm{\mu'})}(0;\bm{\bm{\eta}})|^2
  \end{align}
So we may assume that $\bm{\mu}$ is already of such a form that $\mu_{\iota}$ does not depend on the last $k-l$ components of $\iota$.

Set $\iota^{(l)}= (\iota_k,\cdots,\iota_{l+1})$. 
Then unless $l=0$ (and we may suppose it does not), we then have that  $\bm{\mu_{0\iota'}}$ and $\bm{\mu_{1\iota'}}$ depend on at most $l-1$ components of $\iota'$, so that
\begin{align}\label{star}
\noindent \text{the components of $\bm{\nu}:=\left\{\mu_{0\iota'}\right\}_{\iota\in\{0,1\}^{k}}$ depend on only $l-1$ components of $\iota$}
\end{align}

According to Lemma \ref{thm:CS11} 

\begin{align}\label{wew1}
 &\sum_{\bm{\bm{\eta}}\in\Z^k,s_j>N} |\widehat{\tri^k(\bm{\mu})}(0;\bm{\bm{\eta}})|^2\\\notag{}=&
 \sum_{\bm{\bm{\eta}}\in\Z^k} \widehat{\tri^{k}_{s_j>N}(\bm{\mu_0},\bm{\mu_0})}(0;\bm{\bm{\eta}})\overline{\widehat{\tri^k(\bm{\mu_1},\bm{\mu_1})}(0;\bm{\bm{\eta}})}
\end{align}
Writing, as in (\ref{star}, $\bm{\nu}$ for $(\bm{\mu_0},\bm{\mu_0})$, we then have the bound 
\begin{align}
 &[\sum_{\bm{\bm{\eta}}\in\Z^k} |\tri^{k}_{s_j>N}(\bm{\nu})(0;\bm{\bm{\eta}})|^2]^{\half}[\sum_{\bm{\bm{\eta}}\in\Z^k} |\tri^{k}(\bm{\mu_1},\bm{\mu_1})(0;\bm{\bm{\eta}})|^2]^{\half}\leq&\label{wew2}
A^{2^{k+1}}[\sum_{\bm{\bm{\eta}}\in\Z^k} |\tri^{k}_{s_j>N}(\bm{\nu})(0;\bm{\bm{\eta}})|^2]^{\half}
\end{align}
  since by the Gowers-Cauchy-Schwarz inequality we have
  \begin{align*}
  \langle\bm{\mu_1},\bm{\mu_1}\rangle\leq \prod_{\iota\in\{0,1\}^{k+1}}\|\mu_{\iota}\|_{U^{k+1}}\leq A^{2^{k+1}}.
  \end{align*}
  
  Since $\mu_{0\iota'}=\mu_{1\iota'}$, when we apply Lemma \ref{thm:CS12} we get
  \begin{align}\label{wew3}
   (\ref{wew2}) =&A^{2^{k+1}}\left[\sum_{\bm{\bm{\eta}}\in\Z^k,s_j>N} {\widehat{\tri^{k}_{}(\bm{\nu})(0;\bm{\bm{\eta}})}}\overline{\widehat{\tri^{k}_{s_j+\eta_j>N}(\bm{\nu})}(0;\bm{\bm{\eta}})}\right]^{\half}
  \end{align}

  Applying Cauchy-Schwarz to this, we have that
  
    \begin{align}\label{wew4}
   (\ref{wew3}) \leq&A^{2^{k+1}}\bigg[\sum_{\bm{\bm{\eta}}\in\Z^k,s_j>N} |\widehat{\tri^{k}_{s_j-\eta_j>N}(\bm{\nu})}(0;\bm{\bm{\eta}})|^2\bigg]^{\frac{1}{4}}\bigg[\sum_{\bm{\bm{\eta}}\in\Z^k,s_j>N} |\widehat{\tri^{k}_{}(\bm{\nu})}(0;\bm{\bm{\eta}})|^2\bigg]^{\frac{1}{4}}
 =:A^{2^{k+1}}P_1\cdot P_2\end{align}
  using Lemma \ref{thm:overgrowth} to bound $P_1$, we have
  
      \begin{align}\label{wew5}
   P_1\leq&C'|\langle\bm{\nu}\rangle|^{\frac{1}{4}}\leq C'
  \end{align}
  According to our inductive hypothesis, since the components $\nu_{\iota}$ of $\bm{\nu}=(\bm{\mu_0},\bm{\mu_0})$ depend on only $l-1$ 
  components of $\iota$ as stated in (\ref{star}),  
  (\ref{wew3}) is bounded then by $C''[ \eps^{{2^{-2(l-2)}}}]^{\frac{1}{4}}$, or in other words, we have shown that
  \begin{align}
   (\ref{wew1})\leq C \eps^{{2^{-2(l-1)}}}
  \end{align}
 
  Thus by induction,  (\ref{thirteyninenine}) holds for $\nu$ with the stated properties regardless of what $l$ is; since the largest that $l$ can be is $k$,
  (\ref{thirteynineone}) follows.

  \end{proof}

We may now prove the main result of this subsection, a relative triangle inequality.

\begin{prop}\label{thm:reltri}
 Let $\mu_1,\mu_2\in U^{k+1}$ with $\|\mu_i\|_{U^{k+1}}\leq A$, $\|\mu_{i}\|_{U^{k+1},>N} < \eps$. Then
 
 \begin{align}
  \|\mu_{1}+\mu_2\|_{U^{k+1},>kN}\leq C(A)\eps^{\frac{1}{2^{3k-2}}}
 \end{align}

\end{prop}
\begin{proof}
We expand the expression, obtaining

 \begin{align}
  \|\mu_{1}+\mu_2\|_{U^{k+1},>kN}^{2^k} \\\notag{}=&
  \sum_{\bm{\mu},\mu_{\iota}\in\{\mu_1,\mu_2\}} \tri^{k+1}_{>kN}(\bm{\mu})(\mathbb{T}\times\T^{k+1})\\\notag{}\leq&
    C \sum_{\bm{\mu},\mu_{\iota}\in\{\mu_1,\mu_2\}} \tri^{k+1}_{s_{\infty}>N}(\bm{\mu})(\mathbb{T}\times\T^{k+1})
 \end{align}
Applying Proposition \ref{thm:highcross}, this is bounded by $C(A) \eps^{{2^{-2{(k-1)}}}}$.
\end{proof}


\Ss{Truncations}\label{ch:Truncations}

Our primary goal in this section is Lemma \ref{thm:CS12}.
%

Recall that in Section \ref{ch:Norm Decomposition}, we fixed an approximate identity $\phi$ on $\mathbb{T}^d$ 
such that $\phi_n$  has compact support on the Fourier side for each $n\in\N$, and such that $\widehat{\phi_n}\approx 1_{B(0,2^n)}$.
We then set $\phi_n^c = 1-\phi_n$.

Let $\mu_{\iota}\in U^{k}$, $\iota\in\{0,1\}^{k}$. 
Define $ g\star_{j} f$ to denote the partial convolution $\int g(y)f(x;u_1,\dots,u_{j-1},u_{j} -y,u_{j+1},\dots,u_{k-1}\,dx\,dy\,du$ for $j<k$ and $\star_k = \star$ to denote
convolution on the $x$ variable above. We define
\[ \tri^k_{s_j>N}(\bm{\mu}) = \tcal{\phi_N^c\star_{j}\tri^{k-1}(\bm{\mu_0})}{\tri^{k-1}(\bm{\mu_1})}.\]
and
\[ \tri^k_{s_j-\eta_j>N,s_j>N}(\bm{\mu}) = \tcal{\phi_N^c\star_{j}\tri^{k-1}(\bm{\mu_0})}{\phi_N^c\star_{j}\tri^{k-1}(\bm{\mu_1})}\]

Since $\tn{|\phi_N^c\star_{j}\tri^{k-1}(\bm{\mu_i})|}{|\phi_N^c\star_{j}\tri^{k-1}(\bm{\mu_i})|} <\infty$,
one has from \fouriertransform of \GSM the following

\begin{lemma}\label{thm:rearrange1}
 Let $\mu_{\iota}\in U^k$, $\iota\in\{0,1\}^k$, $N\in\N$, and $1\leq j\leq k$. 
 Then  $\tri^k_{s_j-\eta_j>N,s_j>N}(\bm{\mu})$ and $ \tri^k_{s_j>N} (\bm{\mu})$  exist and for $j\leq k-1$
 
 \begin{align*}
  &\widehat{\tri^k_{s_j-\eta_j>N,s_j>N} (\bm{\mu})}(\xi;\bm{\eta}) = \sum_{\bm{c}\in\Z^{k-1}}\widehat{\phi_N^c}(c_j)\overline{\widehat{\phi_N^c}(c_j-\eta_j)}\widehat{\tri^{k-1}(\bm{\mu_0})}(\xi+\eta_k;\bm{c})\overline{\widehat{\tri^{k-1}(\bm{\mu_1})}(\eta_{k};\bm{c-\eta'})}
\\&\widehat{\tri^k_{s_j>N} (\bm{\mu})}(\xi;\bm{\eta}) = \sum_{\bm{c}\in\Z^{k-1}}\widehat{\phi_N^c}(c_j)\widehat{\tri^{k-1}(\bm{\mu_0})}(\xi+\eta_k;\bm{c})\overline{\widehat{\tri^{k-1}(\bm{\mu_1})}(\eta_{k};\bm{c-\eta'})}
\\&\widehat{\tri^k_{s_k-\eta_k>N,s_k>N} (\bm{\mu})}(\xi;\bm{\eta}) = \sum_{\bm{c}\in\Z^{k-1}}\widehat{\phi_N^c}(\xi+\eta_k)\overline{\widehat{\phi_N^c}(\eta_k)}\widehat{\tri^{k-1}(\bm{\mu_0})}(\xi+\eta_k;\bm{c})\overline{\widehat{\tri^{k-1}(\bm{\mu_1})}(\eta_{k};\bm{c-\eta'})}
\\&\widehat{\tri^k_{s_k>N} (\bm{\mu})}(\xi;\bm{\eta}) = \sum_{\bm{c}\in\Z^{k-1}}\widehat{\phi_N^c}(\xi+\eta_k)\widehat{\tri^{k-1}(\bm{\mu_0})}(\xi+\eta_k;\bm{c})\overline{\widehat{\tri^{k-1}(\bm{\mu_1})}(\eta_{k};\bm{c-\eta'})}
 \end{align*} 
Further, each of these sums is uniformly absolutely convergent in $\bm{\eta}'$.
\end{lemma}`

\begin{lemma}\label{thm:CS11}
 Let $\mu_{\iota}$, $\iota\in\{0,1\}^{k}$ be measures in $U^{k+1}$. Then the following identity holds.
 
 \begin{align}
  &\sum_{\eta\in\Z^{k},s_j>N} |\widehat{\tri^{k}( \bm{\mu})}(0;\eta)|^2\\\notag{}=&
    \sum_{\eta\in\Z^{k},} \widehat{\tri^{k}_{s_j>N}( \bm{\mu_0},\bm{\mu_0})}(0;\eta)\overline{\widehat{\tri^k_{}( \bm{\mu_1},\bm{\mu_1})}(0;\eta)}
   \end{align}

\end{lemma}
\begin{proof}
 The only thing we will do is apply Proposition \itemsaiphoneB and rearrange terms. To start, by that Proposition
 
  \begin{align}
  &\sum_{\eta\in\Z^k,s_j>N} |\widehat{\tri^k( \bm{\mu} )}(0;\eta)|^2\\\notag{}=&
     \sum_{\eta\in\Z^k,s_j>N} \bigg(\sum_a \widehat{\tri^{k-1}(\bm{\mu_1})}(\eta_k;a)\overline{\widehat{\tri^{k-1}(\bm{\mu_0})}(\eta_k;a-\eta')}\bigg)\\\notag{}\cdot&
   \bigg(\sum_c\overline{\widehat{\tri^{k-1}(\bm{\mu_1})}(\eta_k;c)}\widehat{\tri^{k-1}(\bm{\mu_0})}(\eta_k;c-\eta') \bigg)
   \end{align}
 Rearranging the sum, sending $a\mapsto -a+c$ and $\eta'\mapsto -\eta'+c$, and collecting terms, this becomes
 
   \begin{align}
 & \sum_{(a,\eta_k)\in\Z^k} \bigg(\sum_{\eta',s_j>N} \widehat{\tri^{k-1}(\bm{\mu_0} )}(\eta_k;\eta')\overline{\widehat{\tri^{k-1}(\bm{\mu_0})}(\eta_k;\eta'-a)}\bigg)\\\notag{}\cdot&
   \bigg(\sum_c \overline{	\widehat{\tri^{k-1}( \bm{\mu_1})}(\eta_k;c)}\widehat{\tri^{k-1}(\bm{\mu_1} )}(\eta_k;c-a)\bigg)
   \end{align}
 and applying Proposition \itemsaiphoneB once more gives us the desired result.
\end{proof}

\begin{lemma}\label{thm:CS12}
 Let $\mu_{\iota}$, $\iota\in\{0,1\}^{k}$ be signed measures in $U^{k}$ . Then 
 
 \begin{align}\label{nota2}
 &  \sum_{\bm{\eta}\in\Z^k} |\widehat{\tri^k_{s_j>N}( \bm{\mu} )}(0;\bm{\eta})|^2\\\notag{}=&
\sum_{\bm{\eta}\in\Z^k} \tri^k(\bm{\mu_1},\bm{\mu_1})(0;\bm{\eta})\overline{\tri^k_{s_j,s_j-\eta_j>N}(\bm{\mu_0},\bm{\mu_0})(0;\bm{\eta})}
 \end{align}

\end{lemma}

\begin{proof}
 As in the previous lemma, we expand (using the definition of $\widehat{\tri^k_{s_j>N}( \bm{\mu} )}(0;\eta)$)  and
 rearrange the sums using the absolute convergence guaranteed
 by Proposition \itemsaiphoneB  of \GSM and Lemma \ref{thm:rearrange1}. We write the case $j<k$, the case $j=k$ being similar.

 \begin{align}
   (\ref{nota2})=&   \sum_{\eta\in\Z^k}\bigg(\sum_{\bm{a}\in\Z^{k-1},s_j>N}
   \widehat{\tri^{k-1}(\bm{\mu_0})}(\eta_k;\bm{a})\overline{\widehat{\tri^{k-1}(\bm{\mu_1})}(\eta_{k};\bm{a-\eta'})}\\\notag{}\cdot&
   \bigg(\sum_{\bm{c}\in\Z^{k-1}, s_j>N}\overline{\widehat{\tri^{k-1}(\bm{\mu_0})}(\eta_k;\bm{c})}{\widehat{\tri^{k-1}(\bm{\mu_1})}(\eta_{k};\bm{c-\eta'})}\bigg)
 \end{align}

 Sending  $\bm{a}\mapsto -\bm{a}+\bm{c}$ and $\bm{\eta}'\mapsto -\bm{\eta}'+\bm{c}$ makes this

   \begin{align}  
&   \sum_{\bm{\eta},\bm{a},\bm{c}}\widehat{\phi^c_N}(c_j-a_j)\widehat{\phi^c_N}(c_j)\bigg(\widehat{\tri^{k-1}(\bm{\mu_0})}(\eta_k;\bm{c}-\bm{a})\overline{\widehat{\tri^{k-1}(\bm{\mu_1})}(\eta_{k};\bm{\eta'}-\bm{a})}\bigg)\\\notag{}\cdot&
\bigg(\overline{\widehat{\tri^{k-1}(\bm{\mu_0})}(\eta_k;\bm{c})}{\widehat{\tri^{k-1}(\bm{\mu_1})}(\eta_{k};\bm{\eta'})}\bigg)
 \end{align}

 Regrouping, this is
 
   \begin{align*}  
&   \sum_{(\bm{a},\eta_k),s_j>N}
\left(\sum_{\bm{\eta}'}\widehat{\tri^{k-1}(\bm{\mu_1})}(\eta_{k};\bm{\eta'})\overline{\widehat{\tri^{k-1}(\bm{\mu_1})}(\eta_{k};\bm{\eta'}-\bm{a})}\right)\\\cdot&
\bigg(\sum_{\bm{c}}\widehat{\phi_N^c}(c_j)\widehat{\phi_N^c}(c_j-a_j)\overline{\widehat{\tri^{k-1}(\bm{\mu_0})}(\eta_k;\bm{c})}\widehat{\tri^{k-1}(\bm{\mu_0})}(\eta_k;\bm{c}-\bm{a})\bigg)\\=&
\sum_{(\bm{a},\eta_k)} \tri^k(\bm{\mu_0},\bm{\mu_0})(0;\bm{a},\eta_k)\overline{\tri^k_{s_j,s_j-a_j>N}(\bm{\mu_0},\bm{\mu_0})(0;\bm{a},\eta_k)}
 \end{align*}
\end{proof}


\Ss{Key Spatial Estimates}\label{ch:Spatial}

\begin{lemma}\label{thm:overgrowth} Let $\mu_{\iota}$ be measures in $U^k$.

Then   
\begin{align}\label{takeofs}
   &  \sum_{\bm{\eta}\in Z^{k}} |\widehat{\tri^{k}_{s_{j}-\eta_{j}>N,s_{j}>N}(\bm{\mu})}(0;\bm{\eta})|^2 \leq 16 \prod_{\iota\in\{0,1\}^k}\|\mu_{\iota}\|_{U^k}\\&
   \sum|\widehat{\tri^k_{s_{j}-\eta_{j}>N}(\bm{\mu})}(0;\eta)|^2\leq 2\prod_{\iota}\|\mu_{\iota}\|_{U^k}
   \end{align}

\end{lemma}
\begin{proof}
 By \fouriernorm and \gowersnorms of \GSM, both of the above expressions are $\tn{\bm{\nu}}{\bm{\nu}}$ for an appropriate choice of $\bm{\nu}$.
 
 By Lemma \ref{thm:rearrange1}, each such $\bm{\nu} = {\tcal{\bm{\nu_0}}{\bm{\nu_1}}}$  where $\bm{\nu}_i \in \{\bm{\mu_i},\phi_N^c\star_{j}\tri^{k-2}\bm{\mu_i}\}$.
 So both of the above expressions may be expressed further in the form  
 \[\tn{\tcal{\bm{\nu}_0}{\bm{\nu}_1}}{\tcal{\bm{\nu}_0}{\bm{\nu}_1}}. \]
 Recalling the definitions of $\tn{}{},\tcal{}{}$ as well as \factories of \GSM, this takes the form
 \[ \lim_{n_2\ra\infty}\lim_{n_1\ra\infty}\left|\int \prod_{\iota\in\{0,1\}^{k+1}} \mathcal{C}^{|\iota|} \Phi_{n_{2}}\ast \phi_{n_1}\star\phi_{n_1}^{[r]}\ast\bm{\nu}_{\iota}(x-iy;u)\,dx\,dy\,du\right|\]
 for appropriate choices of $\bm{\nu}_{\iota}$, $\Phi_{\iota}$.  Since $\|\phi_N^c\|_{L^{\infty}}\leq 2$, by \phinone followed by Proposition \GCS of \GSM
 we have the claimed bounds.	
\end{proof}

\bibliographystyle{plain}

\bibliography{biblio.bib}

%
%
%

\vskip0.5in

\noindent Marc Carnovale
\\ The Ohio State University \\ 231 W. 18th Ave\\ Columbus Oh, 43210 United States\\ {\em{Email: carnovale.2@osu.edu}}

\end{document}